\documentclass[12pt, reqno, twoside, letterpaper]{amsart}

\usepackage{palindromes_Yu}
\usepackage{algorithm}
\usepackage{algpseudocode}

\title[Represent a natural number as the sum of palindromes in various bases]
{Represent a natural number as the sum of palindromes\break in various bases}
      
\author[Y.\ Gao]{Yu\ Gao}

\address{ACM honored class, 
         Shanghai Jiao Tong University, 
         Shanghai, China.}

\email{xyz2606@sjtu.edu.cn}
        
\date{\today}

\begin{document}

\begin{abstract}
It is shown that the set of palindromes is an additive basis for the natural numbers in any base. Specifically, we prove that every natural number can be expressed as the sum of $O(d)$ palindromes in base $d$.
\end{abstract}

\maketitle


\vskip1.8in

\section{Statement of result}

Let $\NN\defeq\{0,1,2,\ldots\}$ denote the set of natural numbers.
For any integer $d\geq2$,
every number $n\in\NN$ has a unique representation of the form
\begin{equation}
\label{eq:representation}
n=\sum_{j=0}^{l-1}10^j\delta_j,
\end{equation}
where each digit $\delta_j$ belongs to the digit set
$$
\cD\defeq\{0,1,2,\ldots,d-1\},
$$
and the leading digit $\delta_{l-1}$ is nonzero whenever $l\ge 2$.
We use
$$
n=\begin{tabular}{|c|c|c|c|}
\hline
$\delta_{l-1}$&$\cdots$&$\delta_1$&$\delta_0$\\
\hline
\end{tabular}
$$
represents the relation \eqref{eq:representation}. The integer $n$ is said to be a
\emph{palindrome} if its digits satisfy the symmetry condition
$$
\delta_j=\delta_{l-1-j}\qquad(0\le j<l).
$$
Denoting by $P_d$ the collection of all palindromes in $\NN$,
we are to show that $P_d$ is an additive basis for $\NN$ and for any $d\geq2$.

\begin{theorem}
\label{thm:main}
Every natural number is the sum of $O(d)$ palindromes in base $d$ for integer $d\geq2$.
\end{theorem}

\cite{Ban} proved the decimal case using a different approach. Our construction is simpler and more generalized but it uses more palindromes than \cite{Ban}.

\section{The proof}
\label{sec:proof}

\subsection{Notation.}
\label{sec:notate}

We denote by $P_{d,l,k}$ the set of integers $\{c(d^{l-1}+d^k)|0\leq c<d,0\leq k<l-1\}$. 
Next we denote by $\NN _{l,k}$ the set of integers $\{xd^k|d^{l-1-k}\leq x<d^{l-k}\}$. 
In other words, $\NN _{l,k}$ is the set of natural numbers $n$ that have a $d$ base expansion of the form
$$
n=\begin{tabular}{|c|c|c|c|c|c|}
\hline
$\delta_{l-1}$&$\cdots$&$\delta_k$&$0$&$\cdots$&$0$\\
\hline
\end{tabular}
$$
with $\delta_{l-1}\neq 0$.

\subsection{Inductive passage from $\NN _{l,k}$ to $\NN _{l-1,k+1}$.}
\label{sec:induction}
\begin{lemma}
\label{lemma:2d}
Every number $n\in \NN _{l,k}$ is the sum of at most $2d$ numbers given by Algorithm~\ref{algorithm:algo1} from $P_{d,l,k}$ and some number $n'\in \NN _{l-1,k+1}$ for $l\geq k+6$.
\end{lemma}

\begin{algorithm}[htb]
		\caption{Inductive passage from $\NN _{l,k}$ to $\NN _{l-1,k+1}$.}
		\begin{algorithmic}[1]
			\Require an integer $n\in \NN _{l,k}$.
			\Ensure a multiset $S_k$ and the modified $n$.
			\State $S_k \gets \emptyset$
			\State $x \gets (d-1)d^{l-2}+(d-1)d^k$
			\While {$n\geq d^{l-1}+x$}
			\State $n\gets n-x$
			\State $S_k \gets S_k \cup \{x\}$
			\EndWhile
			\State $y \gets (d-1)d^{l-3}+(d-1)d^k$
			\While {$n\geq d^{l-1}$}
			\State $n\gets n-y$
			\State $S_k\gets S_k\cup\{y\}$
			\EndWhile
			\State $c\gets \delta_k(n)$
			\State $z\gets cd^{l-4}+cd^k$
			\State $S_k\gets S_k\cup\{z\}$
			\State $n\gets n-z$
			\State \Return {$S_k$ and $n$}
		\end{algorithmic}
	\label{algorithm:algo1}
\end{algorithm}
\begin{proof}
Since $\NN_{l,k}\subset\{x|d^{l-1}\leq x<d^l\}$, we have
\begin{lemma}
	\label{lem:firstloop}
	The first loop in Algorithm~\ref{algorithm:algo1} executes for at most $d-1$ times.
\end{lemma}
\begin{proof}
	If the first loop executes for at least $d$ times, then $n\geq d^{l-1}+dx\geq d^l$.
\end{proof}
It's clear that $d^{l-1}\leq n<d^{l-1}+x$ after the first loop. Similarly, we have
\begin{lemma}
	\label{lem:secondloop}
	The second loop in Algorithm~\ref{algorithm:algo1} executes for at most $d$ times.
\end{lemma}
\begin{proof}
	If the second loop executes for at least $d+1$ times, then $n\geq d^{l-1}+dy\geq d^{l-1}+x$.
\end{proof}
Thus $|S_k|\leq2d$ for the returned $|S_k|$. To complete the proof, it suffices to show that the returned $n$ by Algorithm~\ref{algorithm:algo1} is in $N_{l-1,k+1}$. $\delta_k(n) = 0$ is satisfied by the choice of $c$. By the end condition of the second loop, 
\begin{align*}
d^{l-1}>n
&\geq d^{l-1}-y-z\\
&\geq d^{l-1}-((d-1)d^{l-3}+(d-1)d^k)-((d-1)d^{l-4}+(d-1)d^k)\\
&=d^k(dd^{l-k-2}-((d-1)d^{l-k-3}+(d-1)d^{l-k-4}+d+(d-2)))\\
&\geq d^k((d-1)d^{l-k-2})\\
&=(d-1)d^{l-2}\\
&\geq d^{l-2}.
\end{align*}
\end{proof}
\subsection{Pseudo-Theorem 1.}
\begin{lemma}
	\label{lemma:2k}
	If $n$ is a natural number with at most $k$ nonzero $d$ base digits, then $n$ is the sum of $2k$ palindromes in base $d$.
\end{lemma}
\begin{proof}
	Let $f$ be the function as following:
\begin{equation}
	f(x)\defeq\begin{cases}
	0&\quad\hbox{if $x=0$};\\
	1&\quad\hbox{if $x=1$};\\
	d-x+1&\quad\hbox{if $x=1,2,\ldots,d-1$}.
	\end{cases}
\end{equation}
So \begin{align*}
	n
	&=\sum_k{d^k\delta_k}\\
	&=\sum_k{f(\delta_k)+(d^k\delta_k-f(\delta_k))}.
\end{align*}
The proof is completed by the observation that $f(\delta_k)$ and $d^k\delta_k-f(\delta_k)$ are both palindromes.
\end{proof}
\begin{lemma}
	\label{lemma:pseudo}
	(Pseudo-Theorem 1)Every number is the sum of $O(d)$ pseudo-palindromes in base $d$ for $d\geq2$.
\end{lemma}
\begin{proof}
	A pseudo-palindrome here is a palindrome with possibly leading zeros.
	For any number $n\in\NN_{l,0}$, using Algorithm~\ref{algorithm:algo1} repeatedly, we have a family of sets $\{S_k|l-k\geq k+6\}$. $S_k$ consists at most $(d-1)$ numbers in $P_{l-k-1,k}$(denoted by $s_{k,j},0\leq j<d$), $d$ numbers in $P_{l-k-2,k}$(denoted by $t_{k,j},0\leq j\leq d$), and 1 number in $P_{l-k-3,k}$(denoted by $r_k$). (We add zeros in case of $|S_k|<2d$.) Notice that $s_{k,j}$ only have two nonzero $d$ base digits, namely $\delta_{l-k-2}$ and $\delta_k$. Considering the form of d-base expansions, 
	$$
	\sum_k{s_{k,j}}=\begin{tabular}{|c|c|c|c|c|}
		\hline
		$\delta_{l-2}(s_{0,j})$&$\delta_{l-3}(s_{1,j})$&$\ldots$&$\delta_1(s_{1,j})$&$\delta_0(s_{0,j})$\\
		\hline
	\end{tabular}
	$$
	is a pseudo-palindrome; so are $\sum_k{t_{k,j}}$  and $\sum_k{r_k}$ . Thus $n$ is the sum of $2d$ pseudo-palindromes and a number $m$ in $\NN_{l',l''}$ for $l'-l''\leq5$. By Lemma~\ref{lemma:2k}, $m$ is also a sum of at most $10$ palindromes.
\end{proof}
\subsection{Proof of Theorem~\ref{thm:main}.}
The only ``bug'' in Pseudo-Theorem 1(Lemma~\ref{lemma:pseudo}) is that the $2d$ pseudo-palindromes may have leading zeros when $0\in S_0$(i.e. when some of $\delta_{l-2}(s_{0,j})$, $\delta_{l-3}(t_{0,j})$ and $\delta_{l-4}(r_j)$ are equal to zero). So we must reduce a natural number $n$ to some number with $0\notin S_0$ in advance.
\begin{lemma}
	\label{lemma:preprocessing}
	Let $n\in\NN_{l,0}$ for $l\geq8$. Then n is the sum of $O(d)$ palindromes produced by Algorithm~\ref{algorithm:algo2} and an natural number $f(n)$ of the form $(d-1)d^{l-2}+(d-1)d^{l-3}+(d-1)d^{l-4}+m$ for $d^2-2d\leq m<d^{l-4}$.
\end{lemma}
	
\begin{algorithm}[htb]
	\caption{Reduce to a proper initial value.}
	\begin{algorithmic}[1]
		\Require an integer $n\in \NN _{l,0}$.
		\Ensure a multiset $T$ and $f(n)$.
		\State $T \gets \emptyset$
		\State $x \gets (d-1)d^{l-2}+(d-1)$
		\While {$n\geq d^{l-1}+x$}
		\State $n\gets n-x$
		\State $T \gets T \cup \{x\}$
		\EndWhile
		\State $y \gets (d-1)d^{l-3}+(d-1)$
		\While {$n\geq d^{l-1}+y$}
		\State $n\gets n-y$
		\State $T\gets T\cup\{y\}$
		\EndWhile
		\State $z \gets (d-1)d^{l-4}+(d-1)$
		\While {$n\geq d^{l-1}+z$}
		\State $n\gets n-z$
		\State $T\gets T\cup\{z\}$
		\EndWhile
		\State $w \gets (d-1)d^{l-5}+(d-1)$
		\While {$n\geq d^{l-1}$}
		\State $n\gets n-w$
		\State $T\gets T\cup\{w\}$
		\EndWhile
		\State \Return {$T$ and $n$(as $f(n)$)}
	\end{algorithmic}
\label{algorithm:algo2}
\end{algorithm}
\begin{proof}
	As the analysis of Algorithm~\ref{algorithm:algo1}, we can easily show that $|T| = O(d)$ and $d^{l-1}-w\leq n<d^{l-1}$ after Algorithm~\ref{algorithm:algo2}.
\begin{align*}
m
&=n-(d-1)d^{l-2}-(d-1)d^{l-3}-(d-1)d^{l-4}\\
&\geq d^{l-4}-w\\
&=d^{l-5}-(d-1)\\
&\geq d^2-2d.
\end{align*}
\end{proof}
We are now ready to give a proof to Theorem~\ref{thm:main}.
If $n\leq d^8+1$, n is the sum of at most $16$ palindromes by Lemma~\ref{lemma:2k}.
Otherwise we use Algorithm~\ref{algorithm:algo1} on $f(n)$ produced by Algorithm~\ref{algorithm:algo2}. In Algorithm~\ref{algorithm:algo1}, both the first and the second loop executes for $(d-1)$ times for $d\geq3$, when the second loop executes for one more time for $d=2$. (The inequality $d^2-2d\leq m<d^{l-4}$ provides that we don't need to care about the trailing digits of $x$, $y$ and $z$ in Algorithm~\ref{algorithm:algo1}.) In the case of $d\geq3$, we modifies S by replacing one of the $y$ produced by the second loop of Algorithm~\ref{algorithm:algo1} with two nonzero palindromes $(d-2)d^{l-3}+(d-2)d^k$ and $d^{l-3}+d^k$. If $c=0$, we just decrease $n$ by $1$($1$ is a palindrome) before simulating Algorithm~\ref{algorithm:algo1}.
Combining the pieces we have $0\notin S_0$ and the pseudo-palindromes become genuine palindromes.

\end{document}